\documentclass[11pt]{article}
 
\usepackage{amssymb,latexsym,epsf,mathrsfs}
\usepackage{amsmath,amscd,amsfonts}

\begin{document}

\newtheorem{theorem}{Theorem}[section]
\newtheorem{lemma}[theorem]{Lemma}
\newtheorem{proposition}[theorem]{Proposition}
\newtheorem{corollary}[theorem]{Corollary}
\newtheorem{definition}[theorem]{Definition}
\newtheorem{example}[theorem]{Example}
\newtheorem{remark}[theorem]{Remark}

\newenvironment{proof}[1][Proof]{\begin{trivlist}
\item[\hskip \labelsep {\bfseries #1}]}{\end{trivlist}} 

\newcommand{\qed}{\nobreak \ifvmode \relax \else
      \ifdim\lastskip<1.5em \hskip-\lastskip
      \hskip1.5em plus0em minus0.5em \fi \nobreak
      \vrule height0.75em width0.5em depth0.25em\fi}

\title{Characterizing finitary functions over non-archimedean RCFs via a topological definition of OVF-integrality}

\author{Yoav Yaffe \\ \footnotesize{yoavy6174@gmail.com} }

\maketitle

\begin{abstract}
When $R$ is a non-archimedean real closed field we say that a function $f\in R(\bar{X})$ is finitary at a point $\bar{b}\in R^n$ if on some neighborhood of $\bar{b}$ the defined values of $f$ are in the finite part of $R$.  In this note we give a characterization of rational functions which are finitary on a set defined by positivity and finiteness conditions.  The main novel ingredient is a proof that OVF-integrality has a natural topological definition, which allows us to apply a known Ganzstellensatz for the relevant valuation.  We also give some information about the Kochen geometry associated with OVF-integrality.
\end{abstract}

\section {Introduction}

Let $R$ be a non-archimedean real closed field, and recall that the {\em finite part} of $R$ is the convex hull of $\mathbb{Z}\subseteq R$.   

\begin{definition}  \label{finitary}
We say that a function $f\in R({\bar X})$ is {\em finitary} at ${\bar b}\in R^n$ if on some neighborhood of $\bar{b}$ its values $f({\bar x})$ are in the finite part of $R$ whenever defined.
\end{definition}
 
The above definition is quoted in \cite{L} (Definition 4.4 there) with the term `bounded', which the author proposed originally.  However in retrospect `finitary' is more appropriate, e.g. note that the collection of functions which are finitary at a given point is not closed under multiplication by an infinite constant.
 
Let ${\bar p}=(p_1,..,p_m)$ be polynomials in $R[{\bar X}]$, let ${\bar g}=(g_1,...,g_l)$ be rational functions in $R({\bar X})$, and define 
\[T=T_{{\bar p},{\bar g}} = \left\{{\bar b} \in R^n\ |\ p_1({\bar b}),\ldots,p_m({\bar b})>0\ \wedge\ g_1,\ldots,g_l \mbox{ are finitary at }{\bar b}\right\} \] 

There is a canonical valuation $v$ on $R$ such that its ring of integers $O_R = \{x\in R\ |\ v(x)\ge 0\}$ equals the finite part of $R$.  We wish to use a Ganzstellensatz on $(R,v)$ (see \cite{LY}, Theorem 7.4) to get a characterization of functions which are finitary on $T$ (i.e. at every point of $T$).  To this end we will need to show that a function is finitary at a point ${\bar b}$ if and only if it is OVF-integral there.  Indeed we'll prove below (Proposition ~\ref{topol}) that for any RCVF the analogous topological definition is equivalent to OVF-integrality, this result being the main contribution of the present paper.  

Note that the proof given in \cite{L} (Lemma 3.6 there) of the harder direction of Proposition ~\ref{topol} seems to be erroneous - clearly the bounded sequence $f(b_n)$ need not have a convergent sub-sequence, but even if it does, the value of the limit $\ell$ might not equal the formula given there.  For example consider the function $f(X,Y) = \frac{X}{X-Y}$, which is not OVF-integral at $(0,0)$.  Let $b_n=([b_n]_X,[b_n]_Y)$, and note that $v([b_n]_X) - \min\{v([b_n]_X),v([b_n]_Y)\} \ge 0$, hence the limit can not equal the negative value $v(\ell)$, contrary to the claim in \cite{L}.  The issue here is that any sequence $b_n$ satisfying $v(f(b_n))<0$ has to converge to $(0,0)$ {\em from a certain direction} (here the line $X=Y$), so one can't assume its points are `generic' relative to the numerator and denominator of $f$, and at any rate the required OVF-valuation near $(0,0)$ has to `point' in the same direction.  See the second part of Remark~\ref{last} for a finer view of the concept of `direction' that we need to consider.

The last step is to note that the set $T$ is defined by OVF-integrality conditions, so is still not first-order definable, while the mentioned Ganzstellensatz deals with the following variant, which is (first-order) defined by naive integrality conditions:
\[S=S_{{\bar p},{\bar g}} = \left\{{\bar b} \in R^n\ |\ p_1({\bar b}),\ldots,p_m({\bar b})>0\ \wedge\ v(g_1({\bar b})),\ldots,v(g_l({\bar b}))\ge 0\right\} \] 
(where $v(g({\bar x}))\ge 0$ is understood to also mean that $g$ is defined at ${\bar x}$).

Since being finitary is equivalent to OVF-integrality, the set $T$ is actually contained in the Kochen closure of $S$ (see \cite{HY}, Remark 3.3), therefore this gap will not pose a real problem.  However by Remark~\ref{last} the set $T$ need not equal this Kochen closure, contrary to what is claimed in \cite{L} (the proof of Corollary 4.6 there).
 
We also discuss briefly the `Kochen geometry' associated with OVF-integrality (see Definition~\ref{kg}), and demonstrate in Remark~\ref{last} its dependence on the ambient variety with some nice examples.

\section {OVF-integrality and the Kochen geometry} 
 
Given a valued field $(K,v)$ we denote its valuation group by $\Gamma_K$, its valuation ring by $O_K = \{a\in K: v(a)\ge 0\}$, and its ideal of `infinitesimals' by ${\cal{M}}_K = \{a\in K: v(a) > 0\}$.  

Recall that an {\em ordered valued field} (or OVF) is a ordered field with a convex non-trivial valuation ring.  The following definitions by Haskell and the author (see \cite{HY}, Subsection 2.2) intend to refine the naive notion of integrality at a point (consider for example whether $f(X,Y) = \frac{X^2}{X^2+Y^2}$ should be integral at $(0,0)$, as opposed to $f(X,Y) = \frac{X}{Y}$ at the same point).

We begin with naming the valuations which interest us in the context of the OVF category.  Let $L$ be a field extension of $K$.  A valuation $\tilde{v}$ on $L$ which extends $v$ is called an {\em OVF-valuation} if there exists some order $\le_L$ on $L$ such that $(L,\tilde{v},\le_L)\models OVF$. 
 
\begin{definition}   \cite{HY}  \label{valuation near point}
Let $(K,v,\le_K)$ be an OVF, $\bar b \in K^n$. We will say that an OVF-valuation $\tilde{v}$ on $L=K({\bar X})$ is {\em near ${\bar b}$} if for every $f\in L$ such that $f({\bar b})=0$ and every $\gamma \in \Gamma_K$ we have $\tilde{v}(f)>\gamma$.
\end{definition}
 
\begin{definition}   \cite{HY}  \label{ovf-integrality}  
Let $(K,v,\le_K)$ be an OVF. Given $f\in L=K(\bar X)$ and $\bar b\in{K}^n$ we say that $f$ is {\em OVF-integral at $\bar b$} if for any OVF-valuation $\tilde{v}$ on $L$ which is near $\bar b$ we have $\tilde{v}(f)\ge 0$. 

For $T \subseteq K^n$ we say that $f$ is {\em OVF-integral on $T$} if $f$ is OVF-integral at ${\bar b}$ for every ${\bar b} \in T$.
\end{definition}

\begin{remark}  \label{conserv}
By existence of OVF-valuations near ${\bar b}$ (see e.g. \cite{HY}, Proposition 4.2) it is easy to conclude that OVF-integrality is equivalent to naive integrality whenever $f$ is defined at ${\bar b}$. 
\end{remark}

We now give a few remarks about the `Kochen geometry' associated with OVF-integrality (see \cite{HY}, Remark 3.3).  
 
\begin{definition}  \cite{HY}  \label{kg}
For any $Q\subseteq K^n$ we define the {\em Kochen closure} of $Q$ to be the set of points ${\bar x}\in K^n$ such that every function $f\in K({\bar X})$ which is OVF-integral on $Q$ is also OVF-integral at ${\bar x}$.  
\end{definition}

Note that the Kochen closure operation is not a `geometry' in the strict sense of the term, i.e. it is not a matroid, since it does not have finite character.  
It is not hard to show that any closed set is also Kochen-closed, however the converse is false - see Remark~\ref{last} for some nice examples.  On the line $K^1$ the converse is true, i.e. a subset of the line is Kochen-closed exactly when it is closed.  Finally note that in dimensions higher than one there is no `Kochen topology', i.e. it is trivial: in $K^2$ for example the union of $V^{int}(\frac{X}{Y})$ and $V^{int}(\frac{Y}{X})$ (using the notation of \cite{HY}, Remark 3.3) equals $K^2 \setminus \{(0,0)\}$.
 
\begin{remark}  \label{ACVF} 
We note in passing that in the context of ACVFs the Kochen closure is {\em not} always contained in the closure, e.g. the Kochen closure of $K^2 \setminus (O_K)^2$ is $K^2$, since in this context the complement of a Kochen-closed set is never bounded.
\end{remark}

\subsection{Topological statement of OVF-integrality for RCVFs}

Recall that a {\em real closed valued field} (or RCVF) is an OVF which is real closed.  The theory RCVF is the model companion of the theory OVF.  In this section we give a topological property which is equivalent to OVF-integrality over an RCVF.

\begin{proposition}  \label{topol}
Let $(K,v,\le_K)$ be a RCVF, let $f\in K({\bar X})$, and let ${\bar b}\in K^n$.  Then $f$ is OVF-integral at ${\bar b}$ if and only if there is some neighborhood $U$ of ${\bar b}$ such that for ${\bar x}\in U$, if $f({\bar x})$ is defined then it is in $O_K$.
\end{proposition}

\begin{proof}
First assume that $f$ is not OVF-integral at ${\bar b}$, i.e. there exists some OVF-valuation $\tilde{v}$ near ${\bar b}$ such that $\tilde{v}(f) < 0$.  Let $\gamma\in\Gamma_K$, and denote the $\gamma$-neighborhood of ${\bar b}$ by $U_\gamma = \{{\bar x} \in K^n\ |\ \bigwedge_{i=1}^n v(x_i-b_i) > \gamma\}$.  We need to find some ${\bar x} \in U_\gamma$ such that $f({\bar x}) \in K\setminus O_K$ (and in particular is defined).  However since $\tilde{v}$ is an OVF-valuation there is some order $\le_L$ on $L=K({\bar X})$ such that $(L,\tilde{v},\le_L)$ is an OVF, and the tuple ${\bar X}\in L^n$ satisfies $\bigwedge_{i=1}^n \big( v(X_i-b_i) > \gamma \big)\ \wedge\ v(f({\bar X}) < 0\ \wedge\ q({\bar X}) \ne 0$ (where $q$ is the denominator of $f$) .  This is a first-order formula over $(K,v,\le_K)$, which is an existentially closed OVF by virtue of being a RCVF (see for example~\cite{LY}, Section 3).  Hence there is some ${\bar x}\in K^n$ satisfying the same formula, as required.
  
For the other (and harder) direction, assume that for every $\gamma\in\Gamma_K$ there is some ${\bar x}_\gamma\in U_\gamma({\bar b})$ such that $f({\bar x}_\gamma)\in K\setminus O_K$.  Now let $P=P({\bar X})$ be the partial type over the valued field $(K,v)$ which says that the tuple ${\bar X}$: (i) is contained in $U_\gamma({\bar b})$ for every $\gamma\in \Gamma_K$; (ii) satisfies $v(f({\bar X}))<0$; (iii) is transcendental over the field $K$; and (iv) satisfies $v(g({\bar X}))\ge 0$ for every $g\in \mathscr{I}_{ord} =  \{\frac{1}{1+r} : r({\bar X})\in K({\bar X}) \mbox{ is a sum of squares} \}$ (see \cite{HY}, the sequel to Lemma 4.1).  

We may assume without loss that ${\bar x}_\gamma$ is transcendental over $K$, by continuity of $f$ (which is defined at ${\bar x}_\gamma$).  It then clearly follows that $\frac{1}{1+r({\bar x}_\gamma)}\in [0,1] \subseteq O_K$.  Therefore $P$ is indeed a partial type, i.e. it is consistent, hence $P$ has a realization ${\bar X}$ in some valued field $(\hat{L},\tilde{v})$ extending $(K,v)$.  Now the restriction of $\tilde{v}$ to $L=K({\bar X})$ is an OVF-valuation (since $\mathscr{I}_{ord}$ has the extension property - see \cite{HY}) which is near ${\bar b}$, and such that $v(f({\bar X}))<0$.  It follows that $f$ is not OVF-integral at ${\bar b}$, as required. $\ \diamond$
\end{proof}   

An interesting conclusion from Proposition~\ref{topol} is that a basic Kochen-closed set (i.e. the OVF-integrality locus $V^{int}(f)$ of a single function $f$) is open.

\section{The relevant ganzstellensatz}

Let $(K,v)$ be any valued field, let $L$ be a field extension of $K$, and assume $A\subseteq L$ is an $O_K$-algebra such that $A \cap K = O_K $.  The set $T=\{1+ma\ :\ m\in {\cal{M}}_K, a\in A\}$ is multiplicative.  The \emph{integral radical} of $A$ in $L$ is the integral closure (in $L$) of the localization $A_T$, and is denoted by $\sqrt[int]{A}$ (see \cite{HY}, Section 2).

Let ${\bar p}=(p_1,..,p_m)$ be polynomials from $K[{\bar X}]$, let ${\bar g}=(g_1,...,g_l)$ be rational functions from $K({\bar X})$, and define:
\[S_{{\bar p},{\bar g}} = \left\{{\bar b} \in K^n\ |\ p_1({\bar b}),\ldots,p_m({\bar b})>0\ \wedge\ v(g_1({\bar b})),\ldots,v(g_l({\bar b}))\ge 0\right\} \]

The following result by Lavi and the author (see \cite{LY}, Theorem 7.4) gives a characterization of rational functions which are OVF-integral on $S_{{\bar p},{\bar g}}$: first, let $Cone({\bar p})$ denote the positive cone generated by the polynomials $\{p_i\ |\ i\}$.  We may assume $S_{{\bar p},{\bar g}}\neq \emptyset$, hence $-1\notin Cone({\bar p})$, and we may define $I_{\bar p} = \{\frac{1}{1 + f}\ |\  f\in Cone({\bar p})\}$.  Finally let $A_{{\bar p},{\bar g}}$ be the $O_K$-algebra generated by $I_{\bar p} \cup \{g_1,...,g_l\}$ in $L=K({\bar X})$. 
 
\begin{theorem}  \cite{LY}  \label{ganz}
Assume that $(K,v,\le_K)$ is a real closed valued field.   

Then for any $h\in L$, $h$ is OVF-integral on $S_{{\bar p},{\bar g}}$ if and only if $h\in \sqrt[int] {A_{{\bar p},{\bar g}}}$. 
\end{theorem}

\section{Finitary functions over non-archimedean RCFs}

Let $R$ be a non-archimedean RCF, and let $v$ be the canonical valuation whose ring of integers equals the finite part of $R$.  Let ${\bar p}=(p_1,..,p_m)$ be polynomials in $R[{\bar x}]$, let ${\bar g}=(g_1,...,g_l)$ be rational functions in $L=R({\bar x})$, and define 
\[T=T_{{\bar p},{\bar g}} = \left\{{\bar b} \in R^n\ |\ p_1({\bar b}),\ldots,p_m({\bar b})>0\ \wedge\ g_1,\ldots,g_l \mbox{ are finitary at }{\bar b}\right\} \] 

Define $I_{\bar p} = \{\frac{1}{1 + f}\ |\  f\in Cone({\bar p})\}$, and let $A_{{\bar p},{\bar g}}$ be the $O_R$-algebra generated by $I_{\bar p} \cup \{g_1,...,g_l\}$ in $L$. 

The pieces are now in place for the following:
     
\begin{theorem}  \label{main}
For any $h\in L$, $h$ is finitary on $T_{{\bar p},{\bar g}}$ if and only if $h\in \sqrt[int] {A_{{\bar p},{\bar g}}}$. 
\end{theorem}

\begin{proof}
By Proposition~\ref{topol} being finitary at a point is equivalent to OVF-integrality, hence we may apply Theorem~\ref{ganz} to the set
\[S=S_{{\bar p},{\bar g}} = \left\{{\bar b} \in R^n\ |\ p_1({\bar b}),\ldots,p_m({\bar b})>0\ \wedge\ v(g_1({\bar b})),\ldots,v(g_l({\bar b}))\ge 0\right\} \]  
and conclude that $h$ is finitary on $S$ if and only if $h\in \sqrt[int] {A_{{\bar p},{\bar g}}}$.  
Clearly $S \subseteq T$, hence it is now sufficient to show that if $h$ is finitary (or equivalently, OVF-integral) on $S$ then it has this property on $T$ as well.  

But every generator of $A_{{\bar p},{\bar g}}$ is OVF-integral on $T_{{\bar p},{\bar g}}$ (for elements of $I_{\bar p}$ by Proposition 4.7 of \cite{LY}, for the $g_i$ by definition), and the collection of functions which are OVF-integral on some set is closed under passing to the generated $O_R$-algebra and taking the integral radical.  Therefore by using the above characterization of functions which are OVF-integral on $S$ we are done. $\ \diamond$
\end{proof}
    
\begin{remark}  \label{last}
The set $T$ above is contained in the Kochen closure of $S$ (see Definition~\ref{kg}), however it need not equal this Kochen closure: consider for example $p(X)=X$, and note that for any RCVF $K$ the Kochen closure of $S_p=\{x\in K\ |\ x>0\}$ equals $\{x\in K\ |\ x\ge 0\}$.  

It is instructive to note here that for RCVFs the Kochen closure of any set $Q$ is contained in the usual closure of $Q$, however they need not be equal.  For example the Kochen closure of $Q=\{(x,0)\in K^2\ |\ x>0\}$ does not contain the point $(0,0)$ (thanks to $f=\frac{Y}{X}$).  The reason that the Kochen closure is sensitive to the ambient variety is that the latter determines the set of possible directions.  

More dramatically, although the Kochen closure of $\{(x,y)\in K^2\ |\ x\ne 0\}$ equals $K^2$, the Kochen closure of the half-plane $H=\{(x,y)\in K^2\ |\ x>0\}$ does not contain the point $(0,0)$, this time thanks to $f=\frac{X}{X+Y^2}$ for example.  This last example gives a better appreciation of the set of possible `directions' that we have in mind.
\end{remark}

Note that the set $T$ defined in Theorem~\ref{main} is actually open, hence one can show more directly that being finitary on $T$ is equivalent to OVF-integrality on $T$: all one needs is the easier direction of Proposition~\ref{topol} and Remark~\ref{conserv} (as done in~\cite{L}, Corollary 3.8).  A similar remark applies to the open set $S$, of course.
 
However there are sets $T'$ which are not open for which there is a Ganzstellensatz, and if one wishes to generalize Theorem~\ref{main} to such sets the full force of Proposition~\ref{topol} seems to be required.  For example, in an unpublished work of Haskell and the author we prove a Ganzstellnsatz for sets defined by equalities, therefore one can also obtain a Ganzstellnsatz for sets defined by weak inequalities (a function is OVF-integral on the union $\{x\ |\ p(x)\ge 0\} = \{x\ |\ p(x) > 0\} \cup \{x\ |\ p(x)= 0\}$ exactly when it is in the intersection of the relevant integral radicals).  If we wish to characterize functions which are finitary on the non-open set
\[T' = \left\{{\bar b} \in R^n\ |\ p_1({\bar b}),\ldots,p_m({\bar b})\ge 0\ \wedge\ g_1,\ldots,g_l \mbox{ are finitary at }{\bar b}\right\} \] 
then we would need to produce suitable OVF-valuations near points on the boundary of $T'$, as done in the proof of Proposition~\ref{topol}.
   


\begin{thebibliography}{9}
      
\bibitem{HY}   D. Haskell and Y. Yaffe, 
Ganzstellens\"atze in theories of valued fields,
{\it J. Math. Logic} {\bf 8} (2008), no. 1, 1--22.
  
\bibitem{L}   N. Lavi, 
Some Positivstellens\"atze in real closed valued fields,
arXiv:1101.5872 (2011).
 
\bibitem{LY}   Y. Yaffe (joint work with N. Lavi), 
A ganzstellensatz for semi-algebraic sets in real closed valued fields,
arXiv:1101.3253 (2011).  
   
\end{thebibliography}
\end{document}